\documentclass[sts]{imsart}

\usepackage{amsmath, amssymb, color}
\usepackage{amsthm} 
\usepackage[colorlinks,citecolor=blue,urlcolor=blue]{hyperref}
\usepackage{xcolor,colortbl}
\usepackage{mathabx}
\usepackage{cleveref, enumitem}
\usepackage{color}
\usepackage{amsfonts, fancybox}
\usepackage{amssymb}
\usepackage[american]{babel}
\usepackage{psfrag,color}
\usepackage{dcolumn}
\usepackage[format=hang, justification=justified, singlelinecheck=false]{subcaption}
\usepackage{flafter, bm, dsfont}
\usepackage[section]{placeins}

\usepackage{multirow}
\usepackage{color}
\usepackage{amsmath}
\usepackage{float}
\usepackage{mathrsfs}
\usepackage{amsfonts}
\usepackage{dsfont}
\usepackage{amssymb}
\usepackage{algorithmic, algorithm}
\usepackage{wrapfig}

\usepackage{amssymb, latexsym}

\usepackage{graphicx}

\numberwithin{equation}{section}

\newtheorem{definition}{Definition}
\newtheorem{theorem}[definition]{Theorem}
\newtheorem{lemma}[definition]{Lemma}

\newtheorem{proposition}[definition]{Proposition}

\newcommand{\cA}{\mathcal{A}}

\newcommand{\cG}{\mathcal{G}}

\newcommand{\cN}{\mathcal{N}}

\newcommand{\cP}{\mathcal{P}}

\newcommand{\cR}{\mathcal{R}}
\newcommand{\cS}{\mathcal{S}}

\newcommand{\cX}{\mathcal{X}}

\newcommand{\R}{{\rm I}\kern-0.18em{\rm R}}
\newcommand{\h}{{\rm I}\kern-0.18em{\rm H}}
\newcommand{\PP}{{\rm I}\kern-0.18em{\rm P}}
\newcommand{\E}{{\rm I}\kern-0.18em{\rm E}}
\newcommand{\Z}{{\rm Z}\kern-0.18em{\rm Z}}
\newcommand{\1}{{\rm 1}\kern-0.24em{\rm I}}
\newcommand{\N}{{\rm I}\kern-0.18em{\rm N}}

\newcommand{\SN}{\mathcal S_{[N]}}

\newcommand{\SSN}{\mathcal S_{[N]}^{++}}

\newcommand{\ELstar}{\E}
\newcommand{\PLstar}{\PP}

\newcommand{\SX}{\mathcal S_{\cX}}
\newcommand{\SSXs}{\mathcal S_{\cX}^{{\tiny +}}}
\newcommand{\SSX}{\mathcal S_{\cX}^{{\tiny ++}}}

\newcommand{\SLambda}{\mathcal S_{\cX}^{\Lambda}}

\newcommand{\DS}{\displaystyle}
\newcommand{\DPP}{\mathsf{DPP}}

\newcommand*\diff{\mathop{}\!\mathrm{d}}
\DeclareMathOperator{\Tr}{Tr}

\DeclareMathOperator{\Diag}{Diag}
\DeclareMathOperator{\diag}{diag}
\DeclareMathOperator{\Var}{Var}

\begin{document}

\begin{frontmatter}

\title{Rates of estimation for determinantal point processes\thanksref{t1}}
\runtitle{MLE of DPP{s}}

\begin{aug}

\author{\fnms{Victor-Emmanuel}~\snm{Brunel}\ead[label=veb]{vebrunel@math.mit.edu}},
\author{\fnms{Ankur}~\snm{Moitra}\thanksref{t3}\ead[label=ankur]{moitra@mit.edu}},
\author{\fnms{Philippe}~\snm{Rigollet}\thanksref{t2}\ead[label=rigollet]{rigollet@math.mit.edu}}
\and
\author{\fnms{John}~\snm{Urschel}\ead[label=urschel]{urschel@mit.edu}}

\affiliation{Massachusetts Institute of Technology}
\thankstext{t1}{Accepted for presentation at Conference on Learning Theory (COLT) 2017}
\thankstext{t2}{This work was supported in part by NSF CAREER DMS-1541099, NSF DMS-1541100, DARPA BAA-16-46 and a grant from the MIT NEC Corporation.}
\thankstext{t3}{This work was supported in part by NSF CAREER Award CCF-1453261, NSF Large CCF-1565235, a David and Lucile Packard Fellowship, an Alfred P. Sloan Fellowship, an Edmund F. Kelley Research Award, a Google Research Award and a grant from the MIT NEC Corporation.}

\address{{Victor-Emmanuel Brunel}\\
{Department of Mathematics} \\
{Massachusetts Institute of Technology}\\
{77 Massachusetts Avenue,}\\
{Cambridge, MA 02139-4307, USA}\\
\printead{veb}
}

\address{{Ankur Moitra}\\
{Department of Mathematics} \\
{Massachusetts Institute of Technology}\\
{77 Massachusetts Avenue,}\\
{Cambridge, MA 02139-4307, USA}\\
\printead{ankur}
}

\address{{Philippe Rigollet}\\
{Department of Mathematics} \\
{Massachusetts Institute of Technology}\\
{77 Massachusetts Avenue,}\\
{Cambridge, MA 02139-4307, USA}\\
\printead{rigollet}
}

\address{{John Urschel}\\
{Department of Mathematics} \\
{Massachusetts Institute of Technology}\\
{77 Massachusetts Avenue,}\\
{Cambridge, MA 02139-4307, USA}\\
\printead{urschel}
}

\runauthor{Brunel et al.}
\end{aug}

\begin{abstract}
Determinantal point processes (DPPs) have wide-ranging applications in machine learning, where they are used to enforce the notion of diversity in subset selection problems. Many estimators have been proposed, but surprisingly the basic properties of the maximum likelihood estimator (MLE) have received little attention. In this paper, we study the local geometry of the expected log-likelihood function to prove several rates of convergence for the MLE. We also give a complete characterization of the case where the MLE converges at a parametric rate. Even in the latter case, we also exhibit a potential curse of dimensionality where the asymptotic variance of the MLE is exponentially large in the dimension of the problem. 
\end{abstract}



\end{frontmatter}

\section{Introduction} \label{SecIntro}

Determinantal point processes (DPPs) describe a family of repulsive point processes; they induce probability distributions that favor configurations of points that are far away from each other. DPPs are often split into two categories: discrete and continuous. In the former case, realizations of the DPP are vectors from the Boolean hypercube $\{0,1\}^N$, while in the latter, they occupy a continuous space such as $\R^d$. In both settings, the notion of distance can be understood in the sense of the natural metric with which the space is endowed. Such processes were formally introduced in the context of quantum mechanics to model systems of fermions (\cite{Mac75}) that were known to have a repulsive behavior, though DPPs have appeared implicitly in earlier work on random matrix theory, e.g.~\cite{Dys62}. Since then, they have played a central role in various corners of probability, algebra and combinatorics (\cite{BorOls00, BorSos03, Bor11, Oko01, OkoRes03}), for example, by allowing exact computations for integrable systems.

Following the seminal work of ~\cite{KulTas12}, both discrete and continuous DPPs have recently gained attention in the machine learning literature where the repulsive character of DPPs has been used to enforce the notion of diversity in subset selection problems. Such problems are pervasive in a variety of applications such as document or timeline summarization (\cite{LinBil12, YaoFanZha16}), image search (\cite{KulTas11,AffFoxAda14}), audio signal processing (\cite{XuOu16}), image segmentation (\cite{LeeChaYan16}), bioinformatics (\cite{BatQuoKul14}), neuroscience (\cite{SnoZemAda13}) and wireless or cellular networks modelization (\cite{MiyShi14, TorLeo14, LiBacDhi15, DenZhoHae15}). DPPs have also been employed as methodological tools in Bayesian and spatial statistics (\cite{KojKom16, BisCoe16}), survey sampling (\cite{LooMar15,ChaJolMar16}) and Monte Carlo methods (\cite{BarHar16}).

Even though most of the aforementioned applications necessitate estimation of the parameters of a DPP, statistical inference for DPPs has received little attention. In this context, maximum likelihood estimation is a natural method, but generally leads to a non-convex optimization problem. This problem has been addressed by various heuristics, including Expectation-Maximization (\cite{GilKulFox14}), MCMC (\cite{AffFoxAda14}), and fixed point algorithms (\cite{MarSra15}). None of these methods come with global guarantees, however. Another route used to overcome the computational issues associated with maximizing the likelihood of DPPs consists of imposing additional modeling constraints, initially in \cite{KulTas12, AffFoxAda14, BarTit15} and, more recently, \cite{DupBac16,GarPaqKoe16,GarPaqKoe16b,MarSra16}, in which assuming a specific low rank structure for the problem enabled the development of sublinear time algorithms.

The statistical properties of the maximum likelihood estimator  for such problems have received attention only in the continuous case and under strong parametric assumptions (\cite{LavMolRub15, BisLav16}) or smoothness assumptions in a nonparametric context (\cite{Bar13}). However, despite their acute relevance to machine learning and several algorithmic advances (see~\cite{MarSra15} and references therein), the statistical properties of general discrete DPPs have not been established. Qualitative and quantitative characterizations of the likelihood function would shed light on the convergence rate of the maximum likelihood estimator, as well as aid in the design of new estimators. 

In this paper, we take an information geometric approach to understand the asymptotic properties of the maximum likelihood estimator. First, we study the curvature of the expected log-likelihood around its maximum. Our main result is an exact characterization of when the maximum likelihood estimator converges at a parametric rate (Theorem~\ref{MainCor}). Moreover, we give quantitative bounds on the strong convexity constant (Proposition~\ref{PropositionPath}) that translate into lower bounds on the asymptotic variance and shed light on what combinatorial parameters of a DPP control said variance. 

The remainder of the paper is as follows. In Section~\ref{SEC:prelim}, we provide an introduction to DPPs together with notions and properties that are useful for our purposes. In Section~\ref{SEC:geometry}, we study the information landscape of DPPs and specifically, the local behavior of the expected log-likelihood around its global maxima. Finally, we translate these results into rates of convergence for maximum likelihood estimation in Section~\ref{SEC:stat}. Certain results and proofs are gathered in the appendices in order to facilitate the narrative.

\subsubsection*{Notation.}
Fix a positive integer $N$ and define $[N]=\{1,2,\ldots,N\}$. Throughout the paper, $\mathcal X$ denotes a subset of  $[N]$. We denote by $\wp(\cX)$ the power set of $\cX$.

We implicitly identify the set of $|\cX| \times |\cX|$ matrices to the the set of mappings from $\cX \times \cX$ to $\R$. As a result,  we denote by $I_\mathcal X$ the identity matrix in $\R^{\mathcal X\times\mathcal X}$ and we omit the subscript whenever $\mathcal X=[N]$. For a matrix $A\in \R^{\mathcal X\times\mathcal X}$ and $J \subseteq\cX$, denote by $A_J$ the restriction of $A$ to $J \times J$. When defined over $\cX \times \cX$, $A_J$ maps elements outside of $J \times J$ to zero.

Let $\SX$ denote the set of symmetric matrices in $\R^{\mathcal X\times\mathcal X}$ matrices and  denote by $\SLambda$ the subset of matrices in $\SX$ that have eigenvalues in $\Lambda \subseteq \R$. Of particular interest are $\SSXs=\cS_{\cX}^{[0, \infty)}$ and $\SSX=\cS_{\cX}^{(0, \infty)}$, the subsets of positive semidefinite and positive definite matrices respectively.

For a matrix $A \in \R^{\cX \times \cX}$, we denote by $\|A\|_F$, $\det(A)$ and $\Tr(A)$ its Frobenius norm, determinant and trace respectively. We set $\det A_\emptyset=1$ and $\Tr A_\emptyset=0$. Moreover, we denote by $\diag(A)$ the vector of size $|\cX|$ with entries given by the diagonal elements of $A$. If $x \in \R^N$, we denote by $\Diag(x)$ the $N \times N$ diagonal matrix with diagonal given by $x$. 

For $\cA \subseteq \SX$, $k\geq 1$ and a smooth function $f:\mathcal A\to \R$, we denote by $\diff^k f(A)$ the $k$-th derivative of $f$ evaluated at $A\in\mathcal A$. This is a $k$-linear map defined on $\cA$; for $k=1$, $\diff f(A)$ is the gradient of $f$, $\diff^2 f(A)$ the Hessian, etc.

Throughout this paper, we say that a matrix $A \in \SX$ is block diagonal if there exists a partition $\{J_1, \ldots, J_k\}$, $k \ge 1$, of $\cX$ such that $A_{ij}=0$ if $i \in J_a, j \in J_b$ and $a \neq b$. The largest number $k$ such that such a representation exists is called the \emph{number of blocks} of $A$ and in this case $J_1, \ldots, J_k$ are called \emph{blocks} of $A$.

\section{Determinantal point processes and $L$-ensembles}
\label{SEC:prelim}

In this section we gather definitions and useful properties, old and new, about determinantal point processes.

A (discrete) \emph{determinantal point process} (DPP) on $\cX$ is a random variable  $Z \in \wp(\cX)$ with distribution 
\begin{equation} \label{DefDPP}
	\PP[J \subseteq Z]=\det(K_J), \hspace{3mm} \forall \,J\subseteq \mathcal X,
\end{equation}
where $K \in \cS_{\cX}^{[0,1]}$, is called the \emph{correlation kernel} of $Z$. 

If it holds further that $K \in  \cS_{\cX}^{[0,1)}$, then $Z$ is called \emph{$L$-ensemble}  and there exists a matrix $L=K(I-K)^{-1} \in  \cS_{\cX}^{+}$ such that
\begin{equation} \label{DefLEnsemble}
	\PP[Z=J]=\frac{\det(L_J)}{\det(I+L)}, \hspace{3mm} \forall\, J \subseteq \mathcal X,
\end{equation}
Using the multilinearity of the determinant, it is easy to see that   \eqref{DefLEnsemble} defines a probability distribution (see Lemma \ref{keyidentity}). We call $L$ the \textit{kernel} of the $L$-ensemble~$Z$. 

Using the inclusion-exclusion principle, it follows from \eqref{DefDPP} that $\PP(Z=\emptyset)=\det(I-K)$. Hence, a DPP $Z$ with correlation kernel $K$ is an $L$-ensemble if and only if $Z$ can be empty with positive probability.

In this work, we only consider DPPs that are $L$-ensembles. In that setup, we can identify $L$-ensembles and DPPs, and the kernel $L$ and correlation kernel $K$ are related by the identities 
\begin{equation} \label{KtoL}
	L=K(I-K)^{-1}\,,
\qquad 
	K=L(I+L)^{-1}.
\end{equation}

In the rest of this work, we only consider kernels $L$ that are positive definite, namely $L \in \cS_{\cX}^{++}$. We denote by $\textsf{DPP}_{\mathcal X}(L)$ the probability distribution associated with the DPP with kernel $L$ and  refer to $L$ as the \emph{parameter} of the DPP in the context of statistical estimation. If $\mathcal X=[N]$, we drop the subscript and only write $\textsf{DPP}(L)$ for a DPP with kernel $L$ on $[N]$.

The probability mass function~\eqref{DefLEnsemble} of $\textsf{DPP}(L)$ depends only on the principal minors of $L$ and on $\det(I+L)$. In particular, $L$ is not fully identified by $\textsf{DPP}(L)$ and the lack of identifiability of $L$ has been characterized exactly \cite[Theorem 4.1]{Kul12}. Denote by $\mathcal D$ the collection of $N\times N$ diagonal matrices with $\pm 1$ diagonal entries. Then, for $L_1,L_2\in\SSN$, 
\begin{equation} \label{Identifiability} 
	\textsf{DPP}(L_1)=\textsf{DPP}(L_2) \iff \exists D\in\mathcal D, L_2=DL_1D.
\end{equation}
Hence, if $L\in \SSN$, then $\DS \{M\in\SSN:\textsf{DPP}(M)=\textsf{DPP}(L)\}=\{DLD; D\in\mathcal D\}$. It is easy to see that the cardinal of this family is always of the form $2^{N-k}$, for some $k\in\{1,\ldots,N\}$. If $L$ is block diagonal (see Section \ref{SecIntro} for the definition), then $k$ is the number of blocks of $L$. Otherwise, $k=1$ and we say that $L$ is \textit{irreducible}. A DPP with kernel $L$ is called irreducible whenever $L$ is. Next we define a graph associated to $L$ that naturally describes its block structure.

\begin{definition}
Fix $\cX \subseteq[N]$. The \emph{determinantal graph} $\mathcal G_L=(\cX, E_L)$ of a DPP with kernel $L \in \SSX$ is the undirected graph with vertices $\cX$ and edge set $E_L=\big\{\{i,j\}\,:\, L_{i,j}\neq 0\big\}$. If $i,j\in \mathcal X$, write $i\sim_L j$ if there exists a path in $\mathcal G_L$ that connects $i$ and $j$. 
\end{definition}

It is not hard to see that a DPP with kernel $L$ is irreducible if and only if its determinantal graph $\cG_L$ is connected. If $L$ is block diagonal, then its blocks correspond to the connected components of $\cG_L$. Moreover, it follows directly from~(\ref{DefLEnsemble}) that if $Z\sim \DPP(L)$ and $L$ has blocks $J_1, \ldots, J_k$, then $Z\cap J_1, \ldots, Z\cap J_k$ are mutually independent DPPs with kernels $L_{J_1}, \ldots, L_{J_k}$ respectively.

Now that we have reviewed useful properties of DPPs, we are in a position to study the information landscape for the statistical problem of estimating the kernel of a DPP from independent observations.

\section{Information geometry}
\label{SEC:geometry}
\subsection{Definitions}

Our goal is to estimate an unknown kernel $L^* \in \SSN$ from $n$ independent copies of $Z \sim \DPP(L^*)$. In this paper, we study the statistical properties of what is arguably the most natural estimation technique: maximum likelihood estimation.

Let $Z_1,\ldots, Z_n$ be $n$ independent copies of $Z\sim \DPP(L^*)$ for some unknown $L^*\in\SSN$.  The (scaled) log-likelihood associated to this model is given for any $L\in\SSN$,
\begin{equation}
\label{EmpLogLike}
\hat\Phi(L)  = \frac{1}{n}\sum_{i=1}^n \log p_{Z_i}(L)  = \sum_{J\subseteq [N]} \hat p_J\log\det(L_J) - \log\det(I+L)\,,
\end{equation}
where $p_J(L)=\PP[Z=J]$ is defined in~\eqref{DefLEnsemble} and $\hat p_J$ is its empirical counterpart defined by
$$
\hat p_J=\frac{1}{n} \sum_{i=1}^n \1(Z_i=J)\,.
$$
Here $\1(\cdot)$ denotes the indicator function. We denote by $\Phi_{L^*}$ the expected log-likelihood as a function of $L$ (resp. $K$):
\begin{equation} \label{ExpLogLike}
	\Phi_{L^*}(L) = \sum_{J\subseteq [N]} p_J(L^*)\log\det(L_J) - \log\det(I+L).
	\end{equation}

For the ease of notation, we assume in the sequel that $L^*$ is fixed, and write simply $\Phi=\Phi_{L^*}$ and $p_J^*=p_J(L^*)$, for $J\subseteq [N]$.

We now proceed to local study of the function $L \mapsto \Phi(L)$ around $L=L^*$ and show, in turn how this analysis can be turned into rates of estimation using rather standard statistical arguments. Specifically, we give a necessary and sufficient condition on $L^*$ so that $\Phi$ is locally strongly concave around $L=L^*$, i.e., the Hessian of $\Phi$ evaluated at $L=L^*$ is definite negative. 

\subsection{Global maxima}

Note that $\Phi(L)$ is, up to an additive constant that does not depend on $L$, the Kullback-Leibler (KL) divergence between $\DPP(L)$ and $\DPP(L^*)$: 

$$
	\Phi(L) = \Phi(L^*) - \textsf{KL}\left(\textsf{DPP}(L^*),\textsf{DPP}(L)\right), \forall L\in\SSN\,,
$$
where $\textsf{KL}$ stands for the Kullback-Leibler divergence between probability measures. In particular, by the properties of this divergence, $\Phi(L)\leq \Phi(L^*)$ for all $L\in\SSN$, and
$$
	\Phi(L)=\Phi(L^*) 
	\iff \textsf{DPP}(L)=\textsf{DPP}(L^*) 
	 \iff L=DL^*D, \hspace{3mm} \mbox{for some } D\in\mathcal D.
$$
As a consequence, the global maxima of $\Phi$ are exactly the matrices $DL^*D$, for $D$ ranging in $\mathcal D$. The following theorem gives a more precise description of $\Phi$ around $L^*$ (and, equivalently, around each $DL^*D$ for $D\in\mathcal D$).

\begin{theorem} \label{MainThm}

Let $L^*\in\SSN$, $Z\sim\textsf{DPP}(L^*)$ and $\Phi=\Phi_{L^*}$, as defined in \eqref{ExpLogLike}. Then, $L^*$ is a critical point of $\Phi$. Moreover, for any $H \in \SN$, 
\begin{equation*}
\diff^2\Phi(L^*)(H,H)=-\Var[ \Tr((L_Z^*)^{-1} H_Z)].
\end{equation*} 

In particular, the Hessian $\diff^2\Phi(L^*)$ is negative semidefinite.
\end{theorem}
\begin{proof}
Theorem \ref{MainThm} is a direct consequence of Lemma \ref{Derivatives} and identities \eqref{keyEq1} and \eqref{keyEq2}. 
\end{proof}

The first part of this theorem is a consequence of the facts that $L^*$ is a global maximum of a smooth $\Phi$ over the open parameter space $\SSN$. The second part of this theorem follows from the usual fact that the Fisher information matrix has two expressions: the opposite of the Hessian of the expected log-likelihood and the variance of the score (derivative of the expected log-likelihood). We also provide a purely algebraic proof of Theorem \ref{MainThm} in the appendix.

Our next result characterizes the null space of $d^2\Phi(L^*)$ in terms of the determinantal graph $\mathcal G_{L^*}$.

\begin{theorem} \label{MainCor}

Under the same assumptions of Theorem \ref{MainThm}, the null space of the quadratic Hessian map $H\in\SN \mapsto \diff^2\Phi(L^*)(H,H)$ is given by 
\begin{equation}
\label{EQ:defNL}
\mathcal N(L^*)=\left\{H\in\SN\,:\,H_{i,j}=0\ \text{for all } i,j \in [N] \ \text{such that}\ \ i\sim_{L^*}j\right\}\,.
\end{equation}
In particular, $\diff^2\Phi(L^*)$ is negative definite if and only if $L^*$ is irreducible.
\end{theorem}
\begin{proof}

Let $H\in\SN$ be in the null space of $\diff^2\Phi(L^*)$, i.e., satisfy $\diff^2\Phi(L^*)(H,H)=0$. We need to prove that $H_{i,j}=0$ for all pairs $i,j\in [N]$ such that $i\sim_{L^*}j$. To that end, we proceed by (strong) induction on the distance between $i$ and $j$ in $\mathcal G_{L^*}$, i.e., the length of the shortest path from $i$ to $j$ (equal to $\infty$ if there is no such path). Denote this distance by $d(i,j)$.

First, by Theorem \ref{MainThm}, $\Var[ \Tr((L_Z^*)^{-1} H_Z)]=0$ so the random variable $\Tr((L_Z^*)^{-1} H_Z)$ takes only one value with probability one. Therefore since $p_J^*>0$ for all $J \subseteq[N]$ and $\Tr((L_\emptyset^*)^{-1} H_\emptyset) = 0$,   we also have
\begin{equation} \label{ConstTrace}
\Tr([L^*_J]^{-1} H_J) = 0, \quad \forall J \subseteq[N].
\end{equation}

We now proceed to the induction.

If $d(i,j)=0$, then $i=j$ and since $L^*$ is definite positive, $L_{i,i}^*\neq 0$. Thus, using \eqref{ConstTrace} with $J=\{i\}$, we get $H_{i,i}=0$. 

If $d(i,j)=1$, then $L_{i,j}^*\neq 0$, yielding $H_{i,j}=0$, using again \eqref{ConstTrace}, with $J=\{i,j\}$ and the fact that $H_{i,i}=H_{j,j}=0$, established above.

Let now $m\geq 2$ be an integer and assume that for all pairs $(i,j)\in [N]^2$ satisfying $d(i,j) \leq m$, $H_{i,j} = 0$.
Let $i,j\in [N]$ be a pair satisfying $d(i,j) = m+1$. Let $(i,k_1,\ldots,k_m,j)$ be a shortest path from $i$ to $j$ in $\mathcal G_{L^*}$ and let $J=\{k_0,k_1,\ldots,k_m,k_{m+1}\}$, where $k_0=i$ and $k_{m+1}=j$. Note that the graph $\mathcal G_{L_J^*}$ induced by $L_J^*$ is a path graph and that for all $s,t=0,\ldots,m+1$ satisfying $|s-t|\leq m$, $d(k_s,k_t)=|s-t|\leq m$, yielding $H_{k_s,k_t}=0$ by induction. Hence,
\begin{equation} \label{Nolabel1354}
	\Tr\left((L^*_J)^{-1} H_J\right) = 2\left((L^*_J)^{-1}\right)_{i,j}H_{i,j}=0,
\end{equation}
by \eqref{ConstTrace} with $J=\{i,j\}$. Let us show that $\left((L^*_J)^{-1}\right)_{i,j}\neq 0$, which will imply that $H_{i,j}=0$.
By writing $(L^*_J)^{-1}$ as the ratio between the adjugate of $L_J^*$ and its determinant, we have
\begin{equation} \label{ratioDets}
	\left | \left((L^*_J)^{-1}\right)_{i,j} \right |= \left | \frac{\det L^*_{J\setminus\{i\},J\setminus\{j\}}}{\det L^*_J} \right |,
\end{equation}
where $L^*_{J\setminus\{i\},J\setminus\{j\}}$ is the submatrix of $L^*_J$ obtained by deleting the $i$-th line and $j$-th column. The determinant of this matrix can be expanded as
\begin{align}
\label{FormDet}	\det L_{J\setminus\{i\},J\setminus\{j\}}^* & = \sum_{\sigma\in\mathcal M_{i,j}}\varepsilon(\sigma)L_{i,\sigma(i)}^*L_{k_1,\sigma(k_1)}^*\ldots L_{k_m,\sigma(k_m)}^*\,,
	\end{align}
where $\mathcal M_{i,j}$ stands for the collection of all one-to-one maps from $J\setminus\{j\}$ to $J\setminus\{i\}$ and, for any such map $\sigma$, $\varepsilon(\sigma)\in\{-1,1\}$. There is only one term in \eqref{FormDet} that is nonzero: Let $\sigma\in\mathcal M_{i,j}$ for which the product in \eqref{FormDet} is nonzero. Recall that the graph induced by $L_J^*$ is a path graph. Since $\sigma(i)\in J\setminus\{i\}$, $L_{i,\sigma(i)}^*=0$ unless $\sigma(i)=k_1$. Then, $L_{k_1,\sigma(k_1)}^*$ is nonzero unless $\sigma(k_1)=k_1$ or $k_2$. Since we already have $\sigma(i)=k_1$ and $\sigma$ is one-to-one, $\sigma(k_1)=k_2$. By induction, we show that $\sigma(k_s)=k_{s+1}$, for $s=1,\ldots,m-1$ and $\sigma(k_m)=j$. As a consequence, $\det L_{J\setminus\{i\},J\setminus\{j\}}^*\neq 0$ and, by \eqref{Nolabel1354} and \eqref{ratioDets}, $H_{i,j}=0$, which we wanted to prove. 

Hence, by induction, we have shown that if $\diff^2\Phi(L^*)(H,H)=0$, then for any pair $i,j\in[N]$ such that $d(i,j)$ is finite, i.e., with $i\sim_{L^*} j$, $H_{i,j}=0$. \vspace{3mm} 

Let us now prove the converse statement: Let $H\in\SN$ satisfy $H_{i,j}=0$, for all $i,j$ with $i\sim_{L^*} j$. First, using Lemma \ref{SpanNullSpace} with its notation, for any $J\subseteq [N]$ and $j=1,\ldots,k$,
\begin{align*}
	D_J^{(j)}(L_J^*)^{-1}D_J^{(j)} & = \left(D_J^{(j)}L_J^*D_J^{(j)}\right)^{-1}  = (L_J^*)^{-1}
\end{align*}
and 
\begin{align*}
	D_J^{(j)}H_J^{(j)}D_J^{(j)}  = -H_J^{(j)}\,.
\end{align*}

Hence,
\begin{align*}
	\Tr\left((L_J^*)^{-1}H^{(j)}_J\right) & = \Tr\left(D^{(j)}(L_J^*)^{-1}D^{(j)}H^{(j)}_J\right) = - \Tr\left((L_J^*)^{-1}H_J^{(j)}\right)=0\,.
\end{align*}
Summing over $j=1,\ldots,k$ yields 
\begin{equation} \label{NullTraceJ}
	\Tr\left((L_J^*)^{-1}H_J\right)=0.
\end{equation}
In a similar fashion,
\begin{equation} \label{NullTrace}
	\Tr\left((I+L^*)^{-1}H\right)=0.
\end{equation}

Hence, using \eqref{keyEq2},
\begin{align*}
	\diff^2\Phi(L^*)(H,H) & = -\sum_{J\subseteq [N]}p_J^*\Tr^2\left((L_J^*)^{-1}H_J\right)+\Tr^2\left((I+L^*)^{-1}H\right)  = 0,
\end{align*}
which ends the proof of the theorem.     
\end{proof}

It follows from Theorem~\ref{MainCor} that  $\Phi_{L^*}$ is locally strongly concave around $L^*$ if and only if $L^*$ is irreducible since, in that case, the smallest eigenvalue of $-\diff^2\Phi(L^*)$ is positive. Nevertheless, this positive eigenvalue may be exponentially small in $N$, leading to a small curvature around the maximum of $\Phi_{L^*}$. This phenomenon is illustrated by the following example.

Consider the tridiagonal matrix $L^*$ given by:
$$L_{i,j}^* = \begin{cases}
		a \mbox{ if } i=j,\\
		b \mbox{ if } |i-j|=1, \\
		0 \mbox{ otherwise,}
\end{cases}
	$$	
where $a$ and $b$ are real numbers. 

\begin{proposition} \label{PropositionPath}
Assume that $a>0$ and $a^2>4 b^2$. Then, $L^*\in\SSN$ and there exist two positive numbers $c_1$ and $c_2$ that depend only on $a$ and $b$ such that
\begin{equation*}
	0 <  \inf_{H\in\SN:\|H\|_F=1}-\diff^2\Phi(L^*)(H,H)\leq c_1e^{-c_2N}.
\end{equation*}
\end{proposition}
\begin{proof}
Consider the matrix $H \in \SN$ with zeros everywhere but in positions $(1,N)$ and $(N,1)$, where its entries are $1$. Note that $\displaystyle{\Tr\left((L_J^*)^{-1}H_J\right)}$ is zero for all $J\subseteq [N]$ such that $J\neq [N]$. This is trivial if $J$ does not contain both $1$ and $N$, since $H_J$ will be the zero matrix. If $J$ contains both $1$ and $N$ but does not contain the whole path that connects them in $\mathcal G_{L^*}$, i.e., if $J$ does not contain the whole space $[N]$, then the subgraph $\mathcal G_{L_J^*}$ has at least two connected components, one containing $1$ and another containing $N$. Hence, $L_J^*$ is block diagonal, with $1$ and $N$ being in different blocks. Therefore, so is $(L_J^*)^{-1}$ and $\displaystyle{\Tr\left((L_J^*)^{-1}H_J\right)=2\left((L_J^*)^{-1}\right)_{1,N}=0}$.

Now, let $J=[N]$. Then, 
\begin{align}
	\Tr\left((L_J^*)^{-1}H_J\right) & = 2\left((L^*)^{-1}\right)_{1,N} \nonumber \\
	& = 2(-1)^{N+1}\frac{\det(L_{[N]\setminus\{1\},[N]\setminus\{N\}}^*)}{\det L^*} \nonumber \\
	\label{Path1} & = 2(-1)^{N+1}\frac{b^{N-1}}{\det L^*}.
\end{align}
Write $\det L^*=u_N$ and observe that 
\begin{equation*}
	u_k=au_{k-1}+b^2u_{k-2}, \hspace{3mm} \forall k\geq 2
\end{equation*}
and $u_1=a, u_2=a^2-b^2$. Since $a^2>4b^2$, there exists $\mu>0$ such that 
\begin{equation} \label{Toplitz}
	u_k\geq \mu\left(\frac{a+\sqrt{a^2-4b^2}}{2}\right)^k, \hspace{3mm} \forall k\geq 1.
\end{equation}
Hence, \eqref{Path1} yields
\begin{equation*}
	\left|\Tr\left((L_J^*)^{-1}H_J\right)\right| \leq \frac{2}{\mu|b|}\left(\frac{2|b|}{a+\sqrt{a^2-4b^2}}\right)^N,
\end{equation*}
which proves the second part of Proposition \ref{PropositionPath}, since $a+\sqrt{a^2-4b^2}>a>2|b|$. 

Finally note that \eqref{Toplitz} implies that all the principal minors of $L^*$ are positive, hence $L^* \in \SSN$.     
\end{proof}

While the Hessian cancels in some directions $H \in \cN(L^*)$ for any reducible $L^* \in \SSN$, the next theorem shows that the fourth derivative is negative in \emph{any} nonzero direction $H \in \cN(L^*)$ so that $\Phi$ is actually curved around $L^*$ in any direction.


\begin{theorem} \label{FourthOrder}

Let $H\in\mathcal N(L^*)$. Then,
\begin{itemize}
	\item[(i)] $\displaystyle{\diff^3\Phi(L^*)(H,H,H)=0}$;
	\item[(ii)] $\displaystyle{\diff^4\Phi(L^*)(H,H,H,H)=-\frac{2}{3}\Var\left[\Tr\left(((L_Z^*)^{-1}H_Z)^2\right)\right]}\leq 0$;
	\item[(iii)] $\displaystyle{\diff^4\Phi(L^*)(H,H,H,H)=0 \iff H=0}$.
\end{itemize}

\end{theorem}
\begin{proof}
Let $H\in\mathcal N(L^*)$. 
By Lemma \ref{Derivatives}, the third derivative of $\Phi$ at $L^*$ is given by
\begin{equation*}
	\diff^3\Phi(L^*)(H,H,H)=2\sum_{J\subseteq [N]}p_J^*\Tr\left(((L_J^*)^{-1}H_J)^3\right)-2\Tr\left(((I+L^*)^{-1}H)^3\right).
\end{equation*}
Together with \eqref{keyEq3}, it yields
\begin{align*} 
	\diff^3\Phi(L^*)(H,H,H) & = -\frac{2}{3}\Big(\sum_{J\subseteq [N]}p_J^*a_{J,1}^3-a_1^3\Big) +\frac{4}{3}\Big(\sum_{J\subseteq [N]}p_J^*a_{J,2}-a_2\Big) \\
	& \hspace{15mm} +\frac{2}{3}\Big(\sum_{J\subseteq [N]}p_J^*a_{J,1}a_{J,2}-a_1a_2\Big).
\end{align*}
Each of the three terms on the right hand side of the above display vanish because of \eqref{NullTraceJ},  $H\in\mathcal N(L^*)$ and \eqref{NullTrace} respectively. This concludes the proof of~{\it (i)}.

Next, the fourth derivative of $\Phi$ at $L^*$ is given by 
\begin{equation*}
	\diff^4\Phi(L^*)(H,H,H,  H)=-6\sum_{J\subseteq [N]}p_J^*\Tr\big(((L_J^*)^{-1}H_J)^4\big)+6\Tr\big(((I+L^*)^{-1}H)^4\big).
\end{equation*}

Using \eqref{keyEq4} together with \eqref{NullTraceJ}, \eqref{NullTrace} and $\diff^3\Phi(L^*)(H,H,H)=0$, it yields \begin{equation*}
	\diff^4\Phi(L^*)(H,H,H,  H)=-\frac{2}{3}\Big(\sum_{J\subseteq [N]}p_J^*\Tr^2\big((L_J^*)^{-1}H_J)^2\big)-\Tr^2\big(((I+L^*)^{-1}H)^2\big)\Big).
\end{equation*}
Since $H\in\mathcal N(L^*)$, meaning $\diff^2\Phi(L^*)(H,H)=0$, we also have
\begin{equation*}
	\Tr\left(((I+L^*)^{-1}H)^2\right)=\sum_{J\subseteq [N]}p_J^*\Tr\left((L_J^*)^{-1}H_J)^2\right).
\end{equation*}
Hence, we can rewrite $\diff^4\Phi(L^*)(H,H,H,  H)$ as
\begin{equation*}
	\diff^4\Phi(L^*)(H,H,H,  H)=-\frac{2}{3}\big(\ELstar\left[\Tr^2\left((L_Z^*)^{-1}H_Z)^2\right)\right]-\ELstar\left[\Tr\left((L_Z^*)^{-1}H_Z)^2\right)\right]^2\big)\,.
\end{equation*}
This concludes the proof of~{\it (ii)}.

To prove~{\it (iii)}, note first that if $H=0$ then trivially $\diff^4\Phi(L^*)(H,H,H,  H)=0$. Assume now that $\diff^4\Phi(L^*)(H,H,H,  H)=0$, which, in view of~{\it (ii)} is equivalent to ${\Var[\Tr(((L_Z^*)^{-1}H_Z)^2)]= 0}$. Since $\Tr(((L_\emptyset^*)^{-1}H_\emptyset)^2)=0$, and $p_J^*>0$ for all $J \subseteq[N]$,  it yields 
\begin{equation}
\label{FourthOrderConst}
\Tr(((L_J^*)^{-1}H_J)^2)=0\quad \forall \, J \subseteq[N]\,.
\end{equation}

Fix $i,j\in[N]$. If $i$ and $j$ are in one and the same block of $L^*$, we know by Theorem \ref{MainCor} that $H_{i,j}=0$. On the other hand, suppose that $i$ and $j$ are in different blocks of $L^*$ and let $J=\{i,j\}$. Denote by $h=H_{i,j}=H_{j,i}$. Since $L_J^*$ is a $2\times 2$ diagonal matrix with nonzero diagonal entries and $H_{i,i}=H_{j,j}=0$, \eqref{FourthOrderConst} readily yields $h=0$. Hence, $H=0$, which completes the proof of~{\it (iii)}.     
\end{proof}
The first part of Theorem \ref{FourthOrder} is obvious, since $L^*$ is a global maximum of $\Phi$. However, we give an algebraic proof of this fact, which is instructive for the proof of the two remaining parts of the theorem.

\section{Maximum likelihood estimation}
\label{SEC:stat}
Let $Z_1,\ldots,Z_n$ be $n$ independent copies of $Z \sim \DPP(L^*)$ with unknown kernel $L^*\in\SSN$. The maximum likelihood estimator (\textit{MLE}) $\hat L$ of $L^*$ is defined as a maximizer of the likelihood $\hat \Phi$ defined in \eqref{EmpLogLike}. Since for all $L\in\SSN$ and all $D\in\mathcal D$, $\hat\Phi(L)=\hat\Phi(DLD)$, there is more than one kernel $\hat L$ that maximizes $\hat\Phi$ in general. We will abuse notation and refer to any such maximizer as \emph{``the" MLE}. 

We measure the performance of the MLE using the \emph{loss} $\ell$ defined by
$$
	\ell(\hat L,L^*)=\min_{D\in\mathcal D}\|\hat L-DL^*D\|_F\,
$$
where we recall that $\|\cdot\|_F$ denotes the Frobenius norm.

The loss $\ell(\hat L, L^*)$ being a random quantity, we also define its associated \emph{risk}~$\cR_n$ by
$$
\mathcal R_n(\hat L,L^*)=\E\big[\ell(\hat L,L^*)\big],
$$
where the expectation is taken with respect to the joint distribution of the iid observation $Z_1,\ldots,Z_n \sim \DPP(L^*)$. 

Our first statistical result establishes that the MLE is a consistent estimator.
\begin{theorem} \label{Consistency}
	$$\ell(\hat L,L^*)\xrightarrow[n\to\infty]{} 0\,,\qquad \text{in probability.}$$
\end{theorem}
\begin{proof}
Our proof is based on Theorem 5.14 in \cite{Vaa98}. We need to prove that there exists a compact subset $E$ of $\SSN$ such that $\hat L\in E$ eventually almost surely. 
Fix $\alpha,\beta \in (0,1)$ to be chosen later such that  $\alpha<\beta$ and define the compact set  of $\SSN$ as
$$
E_{\alpha,\beta}=\big\{L \in \SSN\,:\, K=L(I+L)^{-1} \in \cS_{[N]}^{[\alpha, \beta]}\big\}\,.
$$

Let $\delta=\min_{J\subseteq [N]} p_J^*$. Since $L^*$ is definite positive, $\delta>0$. Define the event $\mathcal A$ by
$$
\cA=\bigcap_{J \subseteq[N]}\big\{p_J^*\leq 2\hat p_J\leq 3p_J^*\big\}\,.
$$
and observe that if the event $\mathcal A$ is satisfied, then we have $3\Phi(L)\leq 2\hat\Phi(L)\leq \Phi(L)$ simultaneously for all $L\in\SSN$.
In particular, 
\begin{equation} \label{ChainIneq}
	\Phi(\hat L) \geq 2\hat\Phi(\hat L)\geq 2\hat\Phi(L^*)\geq 3\Phi(L^*),
\end{equation}
where the second inequality follows from the definition of the MLE. 

By definition of $\delta$,
\begin{align} 
	\PLstar\left[\mathcal A^{\complement}\right] & = \PLstar\left[\exists J\subseteq [N], \hat p_J<p_J^*/2 \mbox{ or } \hat p_J>3p_J^*/2\right] \nonumber \\
	& = \PLstar\left[\exists J\subseteq [N], |\hat p_J-p_J^*|>p_J^*/2\right] \nonumber \\
	& \leq \PLstar\left[\exists J\subseteq [N], |\hat p_J-p_J^*|>\delta/2\right] \nonumber \\
	& \leq \sum_{J\subseteq [N]}\PLstar\left[|\hat p_J-p_J^*|>\delta/2\right] \nonumber \\
	\label{ProbaEventA} & \leq 2^{N+1}e^{-\delta^2 n/2}\,.
\end{align}
where we used a union bound and Hoeffding's inequality. Observe that $\Phi(L^*)<0$, so we can define $\alpha<\exp(3\Phi(L^*)/\delta)$ and $\beta>1-\exp(3\Phi(L^*)/\delta)$ such that $0<\alpha<\beta<1$. Let  $L\in\SSN\setminus E_{\alpha,\beta}$ and $K=L(I+L)^{-1}$. Then, either {\it (i)} $K$ has an eigenvalue that is less than $\alpha$, or {\it (ii)} $K$ has an eigenvalue that is larger than $\beta$. Since all the eigenvalues of $K$ lie in $(0,1)$, we have that  $\det(K)\leq \alpha$ in case {\it (i)} and $\det(I-K)\leq  1-\beta$ in case  {\it (ii)}. By $(\ref{KtoL})$, it quickly follows that
\begin{align*}
	\Phi(L) = \sum_{J\subseteq [N]} p_J^*\log |\det(K-I_{\bar J})|.
\end{align*}
We observe that each term in this sum is negative. Hence, by definition of $\alpha$ and $\beta$,
$$
	\Phi(L) \leq 
	\left\{
	\begin{array}{ll}
	 p_{[N]}^*\log\alpha \leq \delta \log\alpha  < 3\Phi(L^*)\leq \Phi(\hat L)\, & \text{in case {\it (i)}}\\
	 p_\emptyset^*\log(1-\beta)  \leq \delta \log(1-\beta)  < 3\Phi(L^*)\leq \Phi(\hat L)\, & \text{in case {\it (ii)}}
	 \end{array}
\right.
$$
using \eqref{ChainIneq}. Thus, on $\cA$, $\Phi(L)<\Phi(\hat L)$ for all $L \in\SSN\setminus E_{\alpha,\beta}$. It yields that on this event, $\hat L\in E_{\alpha,\beta}$.

Now, let $\varepsilon>0$. For all $J\subseteq [N]$, $p_J(\cdot)$ is a continuous function; hence, we can apply Theorem 5.14 in \cite{Vaa98}, with the compact set $E_{\alpha,\beta}$. This yields
\begin{align*}
	\PLstar[\ell(\hat L,L^*)>\varepsilon] & \leq \PLstar[\ell(\hat L,L^*)>\varepsilon, \hat L\in E_{\alpha,\beta}]+\PLstar[\hat L\notin E_{\alpha,\beta}] \\
	& \leq \PLstar[\ell(\hat L,L^*)>\varepsilon, \hat L\in E_{\alpha,\beta}]+\left(1-\PLstar[\mathcal A]\right).
\end{align*}
Using Theorem 5.14 in \cite{Vaa98}, the first term goes to zero, and the second term goes to zero by \eqref{ProbaEventA}. This completes the proof.
\end{proof}

Theorem \ref{Consistency} shows that consistency of the MLE holds for all $L^*\in\SSN$. However, the MLE can be $\sqrt n$-consistent only when $L^*$ is irreducible. Indeed, this is the only case when the Fisher information is invertible, by Theorem \ref{MainCor}.

Let $M\in\SN$ and $\Sigma$ be a symmetric, positive definite bilinear form on $\SN$. We write $A\sim\mathcal N_{\SN}(M,\Sigma)$ to denote a Wigner random matrix $A\in \SN$, such that for all $H\in\SN$, $\Tr(AH)$ is a Gaussian random variable, with mean $\Tr(M H)$ and variance $\Sigma(H,H)$.

Assume that $L^*$ is irreducible and let $\hat L$ be the MLE. Let $\hat D\in\mathcal D$ be such that 
$$\|\hat D\hat L\hat D-L^*\|_F=\min_{D\in\mathcal D}\|D\hat L D-L^*\|_F$$ 
and set $\tilde L=\hat D\hat L\hat D$. Recall that by Theorem~\ref{MainCor}, the bilinear operator $\diff^2\Phi(L^*)$ is invertible. Let $V(L^*)$ denote its inverse. Then, by Theorem 5.41 in \cite{Vaa98},
\begin{equation}
	\sqrt n (\tilde L-L^*)=-V(L^*)\frac{1}{\sqrt n}\sum_{i=1}^n\left((L_{Z_i}^*)^{-1}-(I+L^*)^{-1}\right)+\rho_n,
\end{equation}
where $\DS \|\rho_n\|_F\xrightarrow[n\to\infty]{}  0.$
Hence, we have proved the following theorem.

\begin{theorem} \label{AsymNormMLE}
	Let $L^*$ be irreducible. Then, $\tilde L$ is asymptotically normal, with asymptotic covariance operator $V(L^*)$:
	$$\sqrt n (\tilde L-L^*)\xrightarrow[n\to\infty]{} \mathcal N_{\SN}\left(0,V(L^*)\right)\,,
	$$
	where the above convergence holds in distribution.
\end{theorem}

Recall that we exhibited in Proposition~\ref{PropositionPath} an irreducible kernel $L^* \in \SSN$ that is non-degenerate (its entries and eigenvalues are either zero or bounded away from zero) such that $V(L^*)[H,H]\ge c^N$ for some positive constant $c$ and unit norm $H \in \SN$. Together with Theorem~\ref{AsymNormMLE}, it implies that while the MLE $\tilde L$ converges at the parametric rate $n^{1/2}$,  $\sqrt{n}\Tr[{(\tilde L-L^*)^\top H}]$ has asymptotic variance of order at least $c^N$ for some constant $c>1$. It implies that the MLE suffers from a \emph{curse of dimensionality}.

When $L^*$ is not irreducible, the MLE is no longer a $\sqrt{n}$-consistent estimator of $L^*$; it is only $n^{1/6}$-consistent. Nevertheless, in this case, the blocks of $L^*$ may still be estimated at the parametric rate, as indicated by the following theorem.

 If $A\in\R^{N\times N}$ and $J,J'\subseteq [N]$, we denote by $A_{J,J'}$ the $N\times N$ matrix whose entry $(i,j)$ is $A_{i,j}$ if $(i,j)\in J\times J'$ and $0$ otherwise. We have the following theorem.

\begin{theorem} \label{AsympNormMLE}
	Let $L^*\in\SSN$ be block diagonal. Then, for any pair of distinct blocks $J$ and $J'$,
\begin{equation}
	\min_{D\in\mathcal D}\|\hat L_{J,J'}-DL_{J,J'}^*D\|_F=O_{\PLstar}(n^{-1/6})
\end{equation}
and 
\begin{equation}
	\min_{D\in\mathcal D}\|\hat L_{J}-DL_{J}^*D\|_F=O_{\PLstar}(n^{-1/2}),
\end{equation}
where $O_{\PLstar}$ is big-$O$ notation in probability.
\end{theorem}

\begin{proof}
The first statement follows from Theorem 5.52 in \cite{Vaa98}, with $\alpha=4$ and $\beta=1$ (the fact that $\beta=1$ being a consequence of the proof of Corollary 5.53 in \cite{Vaa98}). For the second statement, note that since the DPPs $Z\cap J, J \in \cP$ are independent, each $\hat L_{J}, J \in \cP$ is the maximum likelihood estimator of $L_{J}^*$. Since $L_{J}^*$ is irreducible, the $n^{1/2}$-consistency of $\hat L_{J}$ follows from Theorem \ref{AsymNormMLE}. 
\end{proof}


%
%
%
%
%


\bibliographystyle{alphaabbr}
\bibliography{Biblio}

\appendix

\section{A key determinantal identity and its consequences}

We start this section by giving a key yet simple identity for determinants.

\begin{lemma} \label{keyidentity}
For all square matrices $L\in\R^{N\times N}$,
\begin{equation} \label{keyEq}
	\det(I+L)=\sum_{J\subseteq [N]}\det(L_J).
\end{equation}
\end{lemma}
This identity is a direct consequence of the multilinearity of the determinant. Note that it gives the value of the normalizing constant in \eqref{DefLEnsemble}. Successive differentiations of~\eqref{keyEq} with respect to $L$ lead to further useful identities. To that end, recall that if $f(L)=\log\det(L), L\in\SSN$, then for all $H\in\SN$, the directional derivative along $H$ is given by
$$\diff f(L)(H)=\Tr(L^{-1}H)$$
Differentiating \eqref{keyEq} once over $L\in\SSN$ yields

\begin{equation} \label{keyEq1_0}
	\sum_{J\subseteq [N]}\det(L_J)\Tr(L_J^{-1}H_J)=\det(I+L)\Tr((I+L)^{-1}H), \hspace{3mm} \forall H\in\SN.
\end{equation}

In particular, after dividing by $\det(I+L)$,
\begin{equation} \label{keyEq1}
	\sum_{J\subseteq [N]}p_J(L)\Tr(L_J^{-1}H_J)=\Tr((I+L)^{-1}H), \hspace{3mm} \forall H\in\SN.
\end{equation}

In matrix form, \eqref{keyEq1} becomes
\begin{equation} \label{keyEq1MatrixForm}
	\sum_{J\subseteq [N]}p_J(L)L_J^{-1}=(I+L)^{-1}.
\end{equation}
Here we use a slight abuse of notation. For $J\subseteq [N]$, $L_J^{-1}$ (the inverse of $L_J$) has size $|J|$, but we still denote by $L_J^{-1}$ the $N\times N$ matrix whose restriction to $J$ is $L_J^{-1}$ and which has zeros everywhere else. 

Let us introduce some extra notation, for the sake of presentation. For any positive integer $k$ and $J\subseteq [N]$, define 
$$
 a_{J,k}=\Tr\big((L_J^{-1}H_J)^k\big) \quad \text{and} \quad a_k=\Tr\big(((I+L)^{-1}H)^k\big)\,,
 $$ 
 where we omit the dependency in $H\in\SN$. Then, differentiating again \eqref{keyEq1_0} and rearranging terms yields
\begin{equation} \label{keyEq2}
	\sum_{J\subseteq [N]}p_J(L)a_{J,2}-a_2=\sum_{J\subseteq [N]}p_J(L)a_{J,1}^2-a_1^2,
\end{equation}
for all $H\in\SN.$
In the same fashion, further differentiations yield
\begin{align}
	\sum_{J\subseteq [N]}p_J(L)a_{J,3}-a_3 & = -\frac{1}{3}\Big(\sum_{J\subseteq [N]}p_J(L)a_{J,1}^3-a_1^3\Big)+\frac{2}{3}\Big(\sum_{J\subseteq [N]}p_J(L)a_{J,2}-a_2\Big) \nonumber \\
	 \label{keyEq3} & \hspace{8mm}  +\frac{1}{3}\Big(\sum_{J\subseteq [N]}p_J(L)a_{J,1}a_{J,2}-a_1a_2\Big)
\end{align}
and 
\begin{align} 
	\sum_{J\subseteq [N]}&p_J(L)a_{J,4}-a_4 \nonumber \\
	&= \frac{1}{9}\Big(\sum_{J\subseteq [N]}p_J(L)a_{J,1}^4-a_1^4\Big)-\frac{4}{9}\Big(\sum_{J\subseteq [N]}p_J(L)a_{J,1}^2a_{J,2}-a_1^2 a_2\Big) \nonumber \\
	& -\frac{2}{9}\Big(\sum_{J\subseteq [N]}p_J(L)a_{J,1}a_{J,2}-a_1a_2\Big)+\frac{5}{9}\Big(\sum_{J\subseteq [N]}p_J(L)a_{J,1}a_{J,3}-a_1a_3\Big) \nonumber \\
	\label{keyEq4} & +\frac{1}{9}\Big(\sum_{J\subseteq [N]}p_J(L)a_{J,2}^2-a_2^2\Big) +\frac{4}{9}\Big(\sum_{J\subseteq [N]}p_J(L)a_{J,3}-a_3\Big),
\end{align}
for all $H\in\SN.$

\section{The derivatives of $\Phi$}

Let $L^*\in\SSN$ and $\Phi=\Phi_{L^*}$. In this section, we give the general formula for the derivatives of $\Phi$.

\begin{lemma} \label{Derivatives}
	For all positive integers $k$ and all $H\in\SN$,
\begin{align*}
	& \diff^k\Phi(L^*)(H, \ldots, H) \\
	& \hspace{8mm} = (-1)^{k-1}(k-1)!\left(\sum_{J\subseteq [N]} p_J^*\Tr\left(((L_J^*)^{-1}H_J)^k\right)-\Tr\left(((I+L^*)^{-1}H)^k\right)\right).
\end{align*}
\end{lemma}

\paragraph{Proof\\} This lemma can be proven by induction, using the two following facts. If $f(M)=\log\det(M)$ and $g(M)=M^{-1}$ for $M\in\SSN$, then for all $M\in\SSN$ and $H\in\SN$,
$$\diff f(M)(H)=\Tr(M^{-1}H)$$
and
$$\diff g(M)(H)=-M^{-1}HM^{-1}.$$

\section{Auxiliary lemma}

\begin{lemma} \label{SpanNullSpace}
Let $L^*\in\SSN$ and $\mathcal N(L^*)$ be defined as in~\eqref{EQ:defNL}. Let $H\in\mathcal N(L^*)$. Then, $H$ can be decomposed as  $H=H^{(1)}+\ldots +H^{(k)}$ where for each $j=1,\ldots,k$,  $H^{(j)}\in\SN$ is such that $D^{(j)}H^{(j)}D^{(j)}=-H^{(j)}$, for some $D^{(j)}\in\mathcal D$ satisfying $D^{(j)}L^*D^{(j)}=L^*$.
\end{lemma}
\begin{proof}
Let $H\in\mathcal N(L^*)$. Denote by $J_1,\ldots,J_M$ the blocks of $L^*$ ($M=1$ and $J_1=[N]$ if $L^*$ is irreducible). For $i=1,\ldots,M$, let $D^{(i)}=\Diag(2\chi(J_i)-1)\in\mathcal D$. Hence, $D^{(i)}L^*D^{(i)}=L^*$, for all $i=1,\ldots,k$.

For $i,j\in [k]$ with $i< j$, define
$$
H^{(i,j)}=\Diag(\chi(J_i))H\Diag(\chi(J_j))+\Diag(\chi(J_j))H\Diag(\chi(J_i))\,.
$$

Then, it is clear that
$$
H=\sum_{1\leq i<j\leq M}H^{(i,j)} \qquad \text{and} \qquad
	D^{(i)}H^{(i,j)}D^{(i)}=-H^{(i,j)},\ \forall\, i<j\,.
$$
The lemma follows by renumbering the matrices $H^{(i,j)}$. 
\end{proof}

%



%
%
%
%

%
%
%
%

%
%
%
%
%
%

%
%
%

\end{document}